\documentclass[11pt]{article}

\date{November 2, 2017}
\usepackage{geometry}                % See geometry.pdf to learn the layout options. There are lots.
\usepackage[parfill]{parskip}    % Activate to begin paragraphs with an empty line rather than an indent
\usepackage{graphicx}
\usepackage{amssymb}
\usepackage{amsmath}
\usepackage{amsthm}
\usepackage{mathabx}
\usepackage{commath}
\usepackage{epstopdf}

\usepackage[T1]{fontenc}
\usepackage[utf8]{inputenc}
\usepackage{authblk}

\usepackage[ngerman, german, english]{babel} 
\usepackage{bbm}

\newcommand{\R}{\mathbb{R}} % reelle
\newcommand{\N}{\mathbb{N}} % natuerliche
\newcommand*\diff{\mathop{}\!\mathrm{d}} % d in dx im Integral
\newtheorem{theorem}{Theorem}
\newtheorem{proposition}[theorem]{Proposition}
\newtheorem{corollary}[theorem]{Corollary}
\newtheorem{lemma}[theorem]{Lemma}

\title{Classification of positive solutions to a nonlinear biharmonic equation with critical exponent}

\author[1,2]{Rupert L. Frank\thanks{r.frank@lmu.de}}
\author[1]{Tobias K\"onig\thanks{tkoenig@math.lmu.de}}
\affil[1]{Mathematisches Institut, Ludwig-Maximilians-Universit\"at M\"unchen}
\affil[2]{Department of Mathematics, Caltech}

\begin{document}

\maketitle

\begin{abstract}
For $n \geq 5$, we consider positive solutions $u$ of the biharmonic equation 
\[ \Delta^2 u = u^\frac{n+4}{n-4} \qquad \text{on}\ \R^n \setminus \{0\}
\]
with a non-removable singularity at the origin. We show that $|x|^{\frac{n-4}{2}} u$ is a periodic function of $\ln |x|$ and we classify all periodic functions obtained in this way. This result is relevant for the description of the asymptotic behavior near singularities and for the $Q$-curvature problem in conformal geometry.
\end{abstract}

%%%%%%%%%%%%%%%%%%%%%

\section{Introduction and main results}

In this paper we are interested in positive solutions $u$ of the equation
\begin{equation}
\label{equation on Rn}
\Delta^2 u = u^\frac{n+4}{n-4}
\qquad\text{in}\ \R^n\setminus\{0\}
\end{equation}
for $n\geq 5$. As we will explain later in more detail, this equation serves on one hand as a model problem for higher order equations with critical non-linearity and on the other hand has a concrete meaning in the $Q$-curvature problem in conformal geometry. It is well-known that the absence of the maximum principle for equations involving the bi-Laplacian poses great challenges both on a conceptual and on a technical level. Nevertheless we will succeed here to prove a classification result for positive solutions of \eqref{equation on Rn} which is completely analogous to its second order counterpart.

We will work throughout with classical solutions of \eqref{equation on Rn}, that is, $u\in C^4 (\R^n\setminus\{0\})$. Because of the regularity theory in \cite{UhVi} (which extends that in \cite{ChGuYa} to $n\geq 5$) this is not a restriction.

In a fundamental work \cite{lin} Lin has shown that all solutions $u$ with a \emph{removable singularity} at the origin (so that \eqref{equation on Rn} holds in all of $\R^n$) are given by
\begin{equation}
\label{eq:linsol}
u(x) = c_n \left( \frac{\lambda}{1+ \lambda^2 |x-x_0|^2} \right)^\frac{n-4}{2} \,,
\qquad c_n = \left((n-4)(n-2)n(n+2) \right)^{\frac{n-4}{8}} \,,
\end{equation}
for some $\lambda>0$ and $x_0\in\R^n$. Solutions of the closely related equation $\Delta^2 u = |u|^{\frac{8}{n-4}} u$ in $\R^n$ are, in particular, given by optimizers of the Sobolev inequality
$$
\int_{\R^n} (\Delta u)^2 \diff x \geq \mathcal S_n \left( \int_{\R^n} |u|^\frac{2n}{n-4} \diff x \right)^{\frac{n-4}{n}} \,.
$$
These optimizers were classified in \cite{Lieb} in an equivalent dual formulation and are again given by constant multiples of the functions in \eqref{eq:linsol}. For a classification of positive solutions with removable singularities to the four-dimensional analogue of \eqref{equation on Rn} we refer to \cite{ChYa,lin} and for the higher order case to \cite{wei-xu}. 

In this paper we will be concerned with solutions $u$ of \eqref{equation on Rn} with \emph{non-removable singularities}. It was also shown by Lin \cite{lin} that such solutions are necessarily radial. We pass to logarithmic coordinates (in this context also known as \emph{Emden--Fowler coordinates}) and write
$$
u(x) = |x|^{-\frac{n-4}{2}} v(\ln |x|) \,.
$$
By a short computation we find that equation \eqref{equation on Rn} for $u$ is equivalent to the following ordinary differential equation for $v$,
\begin{equation}
\label{eq ODE}
v^{(4)}- \frac{n(n-4)+8}{2} v''+ \frac{n^2(n-4)^2}{16} v - |v|^\frac{8}{n-4} v = 0
\qquad\text{in}\ \R \,.
\end{equation}
Note that positive solutions $u$ of \eqref{equation on Rn} correspond to positive solutions $v$ of \eqref{eq ODE} and so $|v|^\frac{8}{n-4} v = v^\frac{n+4}{n-4}$. For some of our results, however, we also need to consider not necessarily positive functions $v$ and for such functions \eqref{eq ODE} is the relevant extension. We set
$$
a_0 = \left( \frac{n(n-4)}{4} \right)^\frac{n-4}{4} \,.
$$

Our first main result classifies all positive periodic solutions of \eqref{eq ODE} and describes their shape.\\

\begin{theorem} \label{theorem existence}
\begin{enumerate}
\item[(i)] Let $v \in C^4(\R)$ be a solution of \eqref{eq ODE}. Then $\inf_\R |v| \leq a_0$, with equality if and only if $v$ is a non-zero constant.
\item[(ii)] Let $a \in \left(0, a_0 \right)$. Then there is a unique (up to translations) bounded solution $v\in C^4(\R)$ of \eqref{eq ODE} with minimal value $a$. This solution is periodic, has a unique local maximum and minimum per period and is symmetric with respect to its local extrema. 
\end{enumerate}
\end{theorem}

To state our second main result, we denote by $v_a$ the unique solution to \eqref{eq ODE} obtained from Theorem \ref{theorem existence} by requiring that $v_a(0) = \min_\R v_a = a$. Also, denote by $L_a$ the minimal period of $v_a$. For the constant solution $v_{a_0}\equiv a_0$, we set $L_{a_0}=0$. 

The following theorem provides a classification of positive solutions $u$ of \eqref{equation on Rn} with non-removable singularities in terms of a two-parameter family.\\

\begin{theorem} \label{theorem periodic}
Let $u \in C^4(\R^n \setminus \{0\})$ be a positive solution of \eqref{equation on Rn} whose singularity at the origin is non-removable. Then there are $a \in (0, a_0]$ and $L \in [0, L_a]$ such that
\[ u(x) = |x|^{-\frac{n-4}{2}} v_a(\log |x| + L) \,, \]
where $v_a$ is the solution of \eqref{eq ODE} introduced after Theorem \ref{theorem existence}.
\end{theorem}

This theorem answers an open question raised in \cite{wei} and shows, in particular, that the positivity of the scalar curvature in their conjecture is not necessary.

It is easy to see that as $a\to 0$ one has $L_a\to\infty$ and $v_a(t+ L_a/2) \to c_n (2\cosh t)^{-\frac{n-4}{2}}$. Undoing the logarithmic change of variables we therefore recover the non-singular solution \eqref{eq:linsol} in the limit $a\to 0$.

We believe that Theorems \ref{theorem existence} and \ref{theorem periodic} will have several applications. Firstly, it should be a key step in describing the asymptotic behavior near the origin of positive solutions $u$ of $\Delta^2 u = u^{\frac{n+4}{n-4}}$ in a punctured ball $\{ 0 <|x|<\rho\}$. This would be the fourth order analogue of a celebrated result of Caffarelli--Gidas--Spruck \cite{caffarelli}; see also \cite{korevaar}. Secondly, we believe that our theorems will prove useful in the construction of constant $Q$-curvature metrics with isolated singularities in the spirit of the classical works \cite{schoen,mazzeo-pacard} for the scalar curvature; see \cite{BaRe,GuWeZh} for results in this direction in the fourth order case. For an introduction to the $Q$-curvature problem see, for instance, \cite{HaYa}.

We end this introduction by comparing the statement and proof of Theorems \ref{theorem existence} and \ref{theorem periodic} with their second order counterpart, which concerns positive solutions $u$ of
\begin{equation}
\label{eq:secondorder}
-\Delta u = u^\frac{n+2}{n-2}
\qquad\text{in}\ \R^n\setminus\{0\} 
\end{equation}
for $n\geq 3$. A famous result of Caffarelli--Gidas--Spruck \cite{caffarelli} says that if this equation is valid on all of $\R^n$, then
$$
u(x) = c_n' \left( \frac{\lambda}{1+ \lambda^2 |x-x_0|^2} \right)^\frac{n-2}{2} \,,
\qquad
c_n' = \left( n(n-2) \right)^{\frac{n-2}{4}} \,,
$$
for some $\lambda>0$ and $x_0\in\R^n$. Moreover, they show that if $u$ is a positive solution of \eqref{eq:secondorder} with a non-removable singularity, then $u$ is radial. Using this information, Schoen \cite{schoen2} observed that all solutions can be classified by standard phase-plane analysis. Indeed, setting
$$
u(x) = |x|^{-\frac{n-2}{2}} v(\ln|x|)
$$
one obtains
$$
- v'' + \frac{(n-2)^2}{4} v - v^{\frac{n+2}{n-2}} = 0 
\qquad\text{in}\ \R
$$
and the positive solutions of this equation are given by the constant $(\frac{n-2}{2})^\frac{n-2}{2}$, by the homoclinic solution $c_n' (2\cosh(t+T))^\frac{n-2}{2}$ and by periodic solutions uniquely parametrized, up to translations, by their minimal value in $(0, (\frac{n-2}{2})^\frac{n-2}{2} )$. Moreover, these periodic solutions have a unique local maximum and minimum per period and are symmetric with respect to their local extrema.

Thus, our Theorems \ref{theorem existence} and \ref{theorem periodic} provide exactly the same conclusions as in the second order case. Their proof, however, is considerably more difficult, because the phase `plane' in the fourth order case is four-dimensional. Moreover, solutions to fourth order equations show, in general, a much richer and typically more erratic behavior than solutions to second order equations; see, e.g., the introduction of the textbook \cite{peletier} for examples. To emphasize the structure of our equation we abbreviate
\begin{equation}
\label{eq:defabp}
A =  \frac{n(n-4)+8}{2} \,,
\qquad
B = \frac{n^2(n-4)^2}{16}  \,,
\qquad
p= \frac{n+4}{n-4} \,,
\end{equation}
and
\begin{equation} \label{eq definition of f}
f(v) = |v|^{p-1} v - B v
\end{equation}
and rewrite \eqref{eq ODE} as
\begin{equation}
\label{eq ODE with f}
v^{(4)}- Av'' - f(v) = 0
\qquad\text{in}\ \R \,.
\end{equation}
Of fundamental importance for us is that the coefficients $A$ and $B$ in \eqref{eq ODE} satisfy the inequalities
\begin{equation}
\label{eq:abineq}
A > 0
\qquad\text{and}\qquad
4B < A^2 \,.
\end{equation}
The second of these inequalities guarantees that the characteristic equation $\xi^4 - A \xi^2 + B =0$ associated to the linearization of \eqref{eq ODE with f} around the zero solution has four distinct, real solutions. The picture that has emerged from the analysis of fourth order equations is that under this structural assumption solutions are better behaved than those of general fourth order equations and resemble in some sense solutions to second order equations; see, e.g., \cite{peletier,vdb,BuChTo}. The reason is that certain techniques are available which are reminiscent of the maximum principle. Technically, this better, `second-order'-like behavior can be proved for \emph{bounded} solutions of the equation and for such solutions there are certain substitutes for two-dimensional phase plane arguments (see, in particular, \ref{proposition (v,v')-uniqueness} and \ref{proposition energy ordering}). Parts of our analysis will rely on results of van den Berg \cite{vdb} for bounded solutions, which in turn rely on results of Buffoni--Champneys--Toland \cite{BuChTo}. Our crucial new ingredient, however, which does not appear in these works, is that \emph{global solutions are necessarily bounded} (Lemma \ref{bounded}). We emphasize that boundedness is a non-local property and breaks the local character of the ODE analysis.

Most of our results (except for the explicit expression of the homoclinic solution) hold, mutatis mutandis, for any equation of the form \eqref{eq ODE with f} with $f$ given by \eqref{eq definition of f}, where $p>1$ is arbitrary and $A$ and $B$ are arbitrary subject to \eqref{eq:abineq}.

\textbf{Acknowledgement.} Partial support through US National Science Foundation grant DMS-1363432 (R.L.F.) is acknowledged.

%%%%%%%%%%%%%%%%%%%

\section{Classification of global ODE solutions}

In this section we will classify all bounded solutions $v$ of \eqref{eq ODE with f}. Positivity will not play a role here.

We begin with some preliminary remarks, which we will use several times below. The function $v\mapsto f(v)$ in \eqref{eq definition of f} has exactly three zeros, namely, at $0$ and at $\pm B^{\frac{1}{p-1}} =\pm a_0$. These correspond to exactly three constant solutions. Moreover, if $v(t)$ is a solution to \eqref{eq ODE with f}, then so are the functions
\begin{itemize}
\item $v(-t)$ (because \eqref{eq ODE with f} contains only even-order derivatives),
\item $-v(t)$ (because $f$ is odd) and
\item $v(t+T)$ for any $T \in \R$ (because \eqref{eq ODE with f} is autonomous).
\end{itemize}

We now state the main result of this section.\\

\begin{proposition}\label{proposition symmetry and classification}
Let $v \in C^4(\R)$ be a solution of \eqref{eq ODE with f}.
Then one of the following three alternatives holds:
\begin{enumerate}
\item[(a)] $v \equiv \pm B^{\frac{1}{p-1}}$, or $v \equiv 0$,
\item[(b)] $v(t) = \pm c_n (2 \cosh(t-T))^{-\frac{n-4}{2}}$ for some $T \in \R$ with $c_n$ from \eqref{eq:linsol},
\item[(c)] $v$ is periodic, has a unique local maximum and minimum per period and is symmetric with respect to its local extrema.\\
\end{enumerate}
\end{proposition}

For the proof of this proposition we will need two results, taken from \cite{vdb}, which quantify that bounded solutions to the fourth order equation \eqref{eq ODE with f} behave in some respects similar to solutions of a second order equation. As we pointed out in the introduction, for this it is crucial that the relation $4B<A^2$ is satisfied. The first result is that every \emph{bounded} entire solution $v$ is uniquely determined by only \emph{two} (instead of four) initial values.\\

\begin{proposition}[Theorem 1 in \cite{vdb}] \label{proposition (v,v')-uniqueness} 
Let $v, w \in C^4(\R)$ be bounded solutions of \eqref{eq ODE with f} and suppose that $v(0) = w(0)$ and $v'(0) = w'(0)$. Then $v \equiv w$. 
\end{proposition}

Since this result is of crucial importance for us, we give a (slightly more direct) proof with our notation in an appendix.

In order to state the second result, we introduce 
\[
F(v) = \int_0^v f(s) \diff s = \frac{|v|^{p+1}}{p+1} - \frac{B}{2} v^2
\] 
as well as the following quantity, also referred to as \emph{energy},
\[ 
\mathcal E_v(t) = - v'''(t)v'(t) + \frac{1}{2} \big(v''(t)\big)^2 + \frac{A}{2} \big( v'(t) \big)^2 + F(v(t)) \,. 
\]
Using equation \eqref{eq ODE with f} one easily finds that for every solution $v$ of \eqref{eq ODE with f} $\frac{\diff}{\diff t} \mathcal E_v(t) = 0$, that is, the \emph{energy is conserved}. We emphasize that this conservation is a local property and valid on the maximal interval of existence and does not require any a-priori boundedness assumptions like Proposition \ref{proposition (v,v')-uniqueness} and the following Proposition \ref{proposition energy ordering} and Lemma \ref{lemma limits2}.

The second result says that, as in the second order case, the energy is a parameter which orders bounded solutions in the $(v, v')$-phase plane.\\

\begin{proposition}[Theorem 2 in \cite{vdb}] \label{proposition energy ordering}
Let $v, w  \in C^4(\R)$ be bounded solutions of \eqref{eq ODE with f} with $v(0) = w(0)$ and either $v'(0) > w'(0) \geq 0$ or $v'(0)<w'(0)\leq 0$. Then $\mathcal E_v > \mathcal E_w$.
\end{proposition}

For the proof we refer to \cite{vdb}. The assumption there is satisfied since $4B<A^2$. (Note that no a-priori bound on the solutions is necessary for our $f$.)

Next, we state two lemmas concerning the asymptotic behavior of solutions at infinity.\\

\begin{lemma}[Lemma 4 in \cite{vdb}] \label{lemma limits2}
Let $v \in C^4(\R)$ be a bounded solution of \eqref{eq ODE with f}. If $v$ is eventually monotone for $t\to\infty$, then 
$$
\lim_{t\to\infty} v(t) \in \{0, \pm B^{\frac{1}{p-1}} \}
\qquad\text{and}\qquad
\lim_{t\to\infty} v^{(k)}(t) =0
\quad\text{for}\ k=1,2,3 \,.
$$
Similarly, if $v$ is eventually monotone for $t\to -\infty$, then 
$$
\lim_{t\to-\infty} v(t)\in \{0, \pm B^{\frac{1}{p-1}} \}
\qquad\text{and}\qquad
\lim_{t\to -\infty} v^{(k)}(t) =0
\quad\text{for}\ k=1,2,3 \,.
$$
\end{lemma}

The following lemma from \cite{wei} shows that equation \eqref{eq ODE with f} does not have a solution which tends to either plus or minus infinity at infinity, that is, solutions that blow up do so in finite time.\\

\begin{lemma}[Lemma 2.1 in \cite{wei}] \label{lemma limits}
Let $v \in C^4(\R)$ be a solution of \eqref{eq ODE with f}. If $a_+ := \lim_{t \to \infty} v(t)\in \R\cup\{\pm\infty\}$ exists, then $a_+ \in \R$. Similarly, if $a_- := \lim_{t \to -\infty} v(t)\in \R\cup\{\pm\infty\}$ exists, then $a_- \in \R$. 
\end{lemma}

This lemma is proved in \cite{wei} for \emph{positive} solutions. An inspection of the proof shows, however, that this positivity is not needed.

After having recalled these results, we now turn to the proof of Proposition \ref{proposition symmetry and classification}.\\

\begin{lemma} \label{lemma symmetry}
Let $v \in C^4(\R)$ be a bounded solution of \eqref{eq ODE with f}.
\begin{enumerate}
\item[(i)] Suppose that $v'(t_0) = 0$ for some $t_0 \in \R$. Then $v$ is symmetric with respect to $t_0$, i.e., for all $t \in \R$, $v(t_0+t) = v(t_0-t)$. 
\item[(ii)] Suppose that $v(t_0) = 0$ for some $t_0 \in \R$. Then $v$ is antisymmetric with respect to $t_0$, i.e., for all $t\in\R$, $v(t_0 - t) = -v(t_0 + t)$.
\end{enumerate}
\end{lemma}

\begin{proof}
$(i)$ Since equation \eqref{eq ODE with f} is autonomous, we may assume $t_0 = 0$. Moreover, if $v$ is a solution, then so is $w(t) := v(-t)$. Thus $v(0) = w(0)$ and, by assumption, $v'(0) = w'(0) = 0$. Proposition \ref{proposition (v,v')-uniqueness} gives $v \equiv w$.

$(ii)$ Again, we may assume $t_0 = 0$. Moreover, if $v$ solves \eqref{eq ODE with f}, then so does $w(t):= -v(-t)$. Since $v(0) = w(0)$ and $v'(0) = w'(0)$, we conclude by Proposition \ref{proposition (v,v')-uniqueness} that $v \equiv w$. 
\end{proof}

Our second lemma shows uniqueness, up to translations, of the positive homoclinic solution. A similar result for $p=2$ appears in \cite{AmTo} with a different proof.\\

\begin{lemma} \label{lemma unique homoclinic}
Let $v,w\in C^4(\R)$ be positive solutions of \eqref{eq ODE with f} with $\lim_{|t| \to \infty} v = \lim_{|t| \to \infty} w = 0$ and $v'(0) = w'(0) = 0$. Then $v \equiv w$.
\end{lemma}

\begin{proof}
Let us first prove that 0 is the only zero of $v'$ and $w'$. Indeed, if $v'$ had another zero at, say, $t_0 > 0$, then by repeated application of Lemma \ref{lemma symmetry} (note that by assumption, $v$ is bounded) we deduce that $v$ must be periodic of period $2 t_0$. In particular $0 < v(0) = v(2 k t_0)$ for all $k \in \N$, which contradicts the assumption that $v(t) \to 0$ as $t \to \infty$. The argument for $w$ is analogous. Hence we must have 
\begin{equation} 
\label{eq homoclinic decreasing} 
v'(t) <0 \qquad\text{and}\qquad w'(t) < 0 \qquad \text{for all}\ t > 0 \,. \end{equation} 

Next, by Lemma \ref{lemma limits2} and by energy conservation,
\begin{equation} \label{eq lemma unique homoclinic}
\mathcal E_v = \lim_{t\to\infty} \mathcal E_v(t) = F(0) = 0
\qquad\text{and}\qquad
\mathcal E_w = \lim_{t\to\infty} \mathcal E_w(t) = F(0) = 0 \,.
\end{equation}

If $v(0) = w(0)$, we are done by Proposition \ref{proposition (v,v')-uniqueness}. 

To complete the proof, let us suppose for contradiction that $v(0) > w(0)$. We claim that this implies that $v > w$ everywhere. Indeed, otherwise there is $t_0 > 0$ such that $v>w$ on $[0,t_0)$ and $v(t_0) = w(t_0)$. Then by \eqref{eq homoclinic decreasing} we infer that $v'(t_0) \leq w'(t_0) < 0$. If $v'(t_0)=w'(t_0)$, then Proposition \ref{proposition (v,v')-uniqueness} implies $v\equiv w$, contradicting $v(0)>w(0)$. If $v'(t_0) < w'(t_0) < 0$, then Proposition \ref{proposition energy ordering} implies $\mathcal E_v > \mathcal E_w$, which contradicts \eqref{eq lemma unique homoclinic}. Hence $v > w$ everywhere. 

We can now derive the desired contradiction. For every $R >0$, we have, using integration by parts and the fact that $v$ and $w$ satisfy \eqref{eq ODE with f},
\begin{align*}
0 & = \int_{-R}^R w(v^{(4)} - A v'' - f(v)) \\
&= b(R) + \int_{-R}^R v(w^{(4)} - A w'' - f(w)) + \int_{-R}^R wv(w^{p-1} - v^{p-1}) \\
&= b(R) + \int_{-R}^R wv(w^{p-1} - v^{p-1}) \,.
\end{align*}
Here, $b(R)$ contains all the boundary terms coming from the integrations by part. By Lemma \ref{lemma limits2} we have $b(R) \to 0$ as $R \to \infty$. But since $\int_{-R}^R wv(w^{p-1} - v^{p-1})$ is a negative and strictly decreasing function of $R$, we obtain a contradiction by choosing $R$ large enough.
\end{proof}

For the concrete values of $A$, $B$ and $p$ in \eqref{eq:defabp} one can compute the homoclinic solution explicitly. We emphasize that this is the only place in the proof of Proposition \ref{proposition symmetry and classification} where the precise form of $A$, $B$ and $p$ enters.\\

\begin{corollary}\label{corollary v is cosh}
Suppose that $v$ is a positive solution of \eqref{eq ODE} with $\lim_{|t| \to \infty} v(t) = 0$. Then there is $T \in \R$ such that
\[ v(t) = c_n (2 \cosh(t-T))^{-\frac{n-4}{2}}, \qquad t \in \R \,, \]
with $c_n$ from \eqref{eq:linsol}.
\end{corollary}

\begin{proof}
A straightforward calculation shows that $w(t) = c_n (2 \cosh(t))^{-\frac{n-4}{2}}$ solves \eqref{eq ODE}. From the assumptions on $v$ it follows that $v$ has a global maximum at some $T \in \R$. Since $v'(T) = 0$, we can apply Lemma \ref{lemma unique homoclinic} to deduce that $v(\cdot + T) = w$.
\end{proof}

The following lemma is one of the key new results in this paper.\\

\begin{lemma}\label{bounded}
Let $v\in C^4(\R)$ be a solution of \eqref{eq ODE with f}. Then $v$ is bounded.
\end{lemma}

\begin{proof}
By replacing $v(t)$ by $v(-t)$, we only need to show that $v$ is bounded on $[0,\infty)$. We consider the set $Z_+ = \{ t \geq 0\, : \, v'(t) = 0 \}$.

If $Z_+$ is bounded (in particular, if it is empty), then $v$ is monotone for large $t$ and thus admits a limit $a_+$ as $t \to \infty$. By Lemma \ref{lemma limits}, $a_+$ is finite and therefore $v$ is bounded on $[0,\infty)$.

We now assume that $Z_+$ is unbounded. Since $F(u) \to \infty$ as $|u| \to \infty$, there is an $R > |v(0)|$ such that $F(u) > \mathcal E_v$ for all $|u| \geq R$. We claim that $|v| < R$ on $[0, \infty)$ which, in particular, implies that $v$ is bounded on $[0,\infty)$. Indeed, by contradiction assume that $M_R:= \{t \geq 0: |v(t)| \geq R  \}$ is non-empty and define $t^* := \inf M_R$. Since $|v(0)| < R$, we must have $t^* > 0$ and $|v(t^*)|=R$. Replacing $v(t)$ by $-v(t)$ if necessary (which does not change the set $Z_+$), we may assume that $v(t^*)=R$. Then also $v'(t^*) \geq 0$. Since $Z_+$ is unbounded, the set $Z_+\cap [t^*,\infty)$ is non-empty and we can set $T := \inf \left( Z_+ \cap [t^*, \infty)\right)$. Then $v'(T)=0$ and $v' \geq 0$ on $[t^*, T]$ by continuity of $v'$. Thus $v(T) \geq v(t^*) = R$, and we deduce that 
\[ \mathcal E_v(T)= \frac{1}{2} v''(T)^2 + F(v(T)) \geq F(v(T)) > \mathcal E_v \,, \]
a contradiction to energy conservation. This completes the proof of Lemma \ref{bounded}.
\end{proof}

We are now ready to prove the main result of this section.

\begin{proof}[Proof of Proposition \ref{proposition symmetry and classification}] Let $v \in C^4(\R)$ be a solution to \eqref{eq ODE with f}, and set 
\[ Z:= \{ t \in \R \,: \, v'(t) = 0\} \,. \]
We distinguish several cases:

Suppose first that $Z = \emptyset$, so $v$ is strictly monotone. We will show that this case cannot occur. Up to replacing $v(t)$ by $v(-t)$, we may assume that $v$ is strictly increasing, and so both limits $a_\pm = \lim_{t \to \pm \infty} v(t)$ exist in $\R\cup\{\pm\infty\}$. By Lemmas \ref{lemma limits} and \ref{lemma limits2}, we are reduced to studying three cases, each of which will lead to a contradiction via an energy argument.

If $a_-= 0$ and $a_+ = B^{\frac{1}{p-1}}$, then using Lemma \ref{lemma limits2} we get $\lim_{t\to-\infty} \mathcal E_v(t) = F(0) = 0$, while $\lim_{t\to+\infty} \mathcal E_v(t) = F(B^{\frac{1}{p-1}}) < 0$, a contradiction to energy conservation. Analogously, a contradiction is obtained if $a_- = -B^{\frac{1}{p-1}}$ and $a_+ = 0$.

It remains to consider the case $a_- = -B^{\frac{1}{p-1}}$, $a_+ = B^{\frac{1}{p-1}}$. Then as above, by Lemma \ref{lemma limits2}, 
\begin{equation} \label{energy less than zero}
\lim_{|t|\to\infty} \mathcal E_v(t) = F(B^{\frac{1}{p-1}})<0 \,.
\end{equation} 
On the other hand, by \cite[Corollary 6]{vdb}, the inequality 
\begin{equation} \label{energy inequality}
\mathcal{E}_v(t) \geq \frac{1}{2}v''(t)^2 + F(v(t))
\end{equation}
holds for all $t \in \R$. But now evaluating the energy at $t_0$ such that $v(t_0) =0$ gives, together with \eqref{energy inequality}, that $\mathcal E_v(t_0) \geq  \frac{1}{2}v''(t_0)^2 + F(0) \geq 0$, in contradiction to \eqref{energy less than zero} and energy conservation. Altogether, we have shown that the case $Z = \emptyset$ cannot occur.

If $|Z| = 1$, we may assume, up to a translation, that $Z = \{0\}$. Then $v$ is strictly monotone on $(-\infty,0)$ and $(0,\infty)$, and so both limits $a_\pm = \lim_{t \to \pm \infty} v(t)$ exist in $\R\cup\{\pm\infty\}$. By Lemmas \ref{lemma limits2} these limits are finite, so $v$ is bounded and, by Lemma \ref{lemma symmetry}, even. Therefore $a_+=a_-$. By Lemma \ref{lemma limits}, only three cases can occur: $a_+ = a_- = 0$ or $a_+ = a_- = \pm B^{\frac{1}{p-1}}$. In the first case, monotonicity implies that either $v >0$ or $v<0$, and we conclude that $v(t) = \pm c_n (2 \cosh(t))^{-\frac{n-4}{2}}$ by Lemma \ref{lemma unique homoclinic}. 

As for the other cases, let us without loss assume that $a_+ = a_- = B^{\frac{1}{p-1}}$ (otherwise replace $v$ by $-v$). We derive a contradiction as follows. Since $v$ is strictly monotone on $[0, \infty)$, $v(0) \neq B^{\frac{1}{p-1}}$, and from $\mathcal E_v(0) = \frac{1}{2}v''^2(0) + F(v(0)) \geq F(B^{\frac{1}{p-1}})$ we infer that $v(0) = -B^{\frac{1}{p-1}}$ (since $F$ takes on its global minimal value only in $ \pm B^{\frac{1}{p-1}}$). Hence $v$ changes sign, i.e. there is $t_0 \in \R$ such that $v(t_0)=0$. By Lemma \ref{lemma symmetry}, $v$ is antisymmetric with respect to $t_0$. But this is a contradiction to the fact that both $a_+$ and $a_-$ are positive. Altogether we have thus shown that if $|Z|=1$, then $v(t) = \pm c_n (2 \cosh(t))^{-\frac{n-4}{2}}$.

Finally, let us consider the case where $|Z| \geq 2$. By continuity of $v'$, we see that unless $v$ is constant (and hence $v \equiv \pm B^{\frac{1}{p-1}}$ or $v\equiv 0$), the set $Z$ cannot be dense, i.e., there are real numbers $c < d$ such that $v'(c) = v'(d) = 0$ and $v' \neq 0$ on $(c,d)$. By Lemma \ref{bounded} $v$ is bounded and therefore we can use Lemma \ref{lemma symmetry} as in the first part of the proof of Lemma \ref{lemma unique homoclinic} to conclude that $v$ must be periodic of period $2(d-c)$. Moreover, since $v$ is strictly monotone on $(c,d)$, there is only one maximum and minimum per period interval, and these are strict. The symmetry with respect to the extrema follows at once from Lemma \ref{lemma symmetry}. This completes the proof of Proposition \ref{proposition symmetry and classification}.
\end{proof}

%%%%%%%%%%%%%%%%%

\section{Proof of the main results}

\subsection{Proof of Theorem \ref{theorem existence}}

We begin with the proof of part $(i)$ of Theorem \ref{theorem existence}. Let $v \in C^4(\R)$ be a solution of \eqref{eq ODE with f}. By Proposition \ref{proposition symmetry and classification}, the only case where 
\begin{equation} \label{inf leq B} \inf_\R |v| \leq B^{\frac{1}{p-1}} \end{equation}
may fail to hold is when $v$ is periodic. In this case, $v$ possesses a local minimum at, say, $t_0 \in \R$. Note that if $v$ has a zero  then \eqref{inf leq B} is automatically fulfilled, so we may assume that $v$ has a fixed sign and, up to replacing $v$ by $-v$, we may assume that $v >0$. But by \cite[Lemma 2.6]{wei}, either $v$ is constant (and hence $v \equiv B^{\frac{1}{p-1}}$) or $v(t_0) < B^{\frac{1}{p-1}}$, so that \eqref{inf leq B} holds with strict inequality.

We turn now to proving part $(ii)$ of Theorem \ref{theorem existence}. We proceed via a shooting argument. The value $a \in (0, B^{\frac{1}{p-1}})$ will be considered to be fixed throughout the following argument. 

For $\beta \geq 0$, we denote by $v_\beta$ the unique solution of \eqref{eq ODE with f} with the initial values
\begin{align} \label{eq IVP}
v(0) = a, \quad v'(0) = 0, \quad v''(0) = \beta, \quad v'''(0) = 0 \,,
\end{align}
and by $T_\beta\in(0,\infty]$ its maximal forward time of existence. Also, let $b := - \min_{v \in \R_+} f(v)$. 

Suppose that $\beta> \frac Ab =:\beta_0$. Then we see from 
\begin{equation} \label{eq IVP RHS} v_\beta^{(4)} = A v_\beta'' +  f(v_\beta) \end{equation}
and \eqref{eq IVP} that $v_\beta^{(4)} >0$ initially. Thus, $v_\beta''$ increases initially, and since the right hand side of equation \eqref{eq IVP RHS} is positive initially, it is easy to see that it will stay positive on $[0,T_\beta)$. Thus, $v_\beta^{(4)} >0$ on $[0,T_\beta)$, which implies that $v_\beta$ and its first three derivatives all keep increasing on $[0,T_\beta)$. Thus, if $T_\beta=\infty$, then $v_\beta$ is unbounded. On the other hand, if $T_\beta<\infty$, then $v_\beta(t)\to \infty$ as $t\to T_\beta$ (since $f$ is locally Lipschitz). To summarize, $v_\beta$ increases monotonically on $[0,T_\beta)$ and diverges to $+\infty$ as $t \to T_\beta$ for $\beta\geq\beta_0$. 

So we can restrict our search to $\beta \in [0, \beta_0]$.  However, for all $\beta \leq \beta_0$, we have the uniform energy bound
\[ \mathcal E_{v_\beta}(0) = \frac{\beta^2}{2} + F(a) \leq \frac{\beta_0^2}{2} + F(a). \]
Since $F(v) \to \infty$ as $v\to\infty$, there is an $R>0$ such that $F(v) > \frac{\beta_0^2}{2} + F(a)$ for all $v > R$. This implies that whenever $\beta \leq \beta_0$ and $v_\beta(t_0) > R$, we must have $v_\beta'(t_0) \neq  0$, for otherwise 
\[ \mathcal E_{v_\beta}(t_0) = \frac{v_\beta''(t_0)^2}{2} + F(v_\beta(t_0)) \geq F(v_\beta(t_0)) > \frac{\beta_0^2}{2} + F(a) \,,\]  
which contradicts the upper bound on $\mathcal E_{v_\beta}(0)$ and energy conservation. In particular, $v_\beta$ which enters the interval $(R, \infty)$ cannot leave it again, and hence is certainly not the periodic solution we are looking for.

On the other hand, if $\beta = 0$, we see from \eqref{eq IVP RHS} that $v_0^{(4)}(0) =  f(a) < 0$, and hence $v_0(t)$ and $v_0''(t)$ are strictly decreasing on some small interval $t \in (0, \sigma)$. Since $f(v)<0$ for $v \in (0, a)$,  we deduce from \eqref{eq IVP RHS} that $v_0^{(k)}(t)$, $k=1,2,3$, stay strictly negative until $v_0(t)$ reaches a negative value. Hence, if $\beta = 0$, there must be $t_0$ such that $v_0(t_0)< 0$. 

All of the previous considerations lead us to defining the following shooting sets,
\begin{align*}
&S:= \{\beta \geq 0: \; v_\beta(t) < 0 \quad \text{for some  } t \in (0,T_\beta) \} \,, \\
& T:= \{ \beta \geq 0: \; v_\beta(t) > R \quad \text{ for some  } t \in (0,T_\beta) \text{  and  } v_\beta >  0 \text{ on } [0, t] \}.
\end{align*}
Clearly, $S$ and $T$ are open in $[0, \infty)$ because of the continuous dependence of the solution on the initial conditions. 
Moreover, $S$ and $T$ are disjoint because, as we observed above, once a solution $v_\beta$ enters the interval $(R, \infty)$, it stays there. We also already argued above that $0 \in S$ and $(\beta_0, \infty) \subset T$, i.e. both $S \neq \emptyset$ and $T \neq \emptyset$. 

Since our shooting parameter interval $[0, \infty)$ is connected, we deduce that $S \cup T \neq [0, \infty)$. Hence there must be $\beta^*>0$ and a corresponding solution $v^* := v_{\beta^*}$ such that $0 \leq v^* \leq R$. In particular, $v^*$ is bounded. This and the fact that $f$ is locally Lipschitz imply that $T_{\beta^*}=\infty$. By even reflection, we obtain a solution defined on all of $\R$, which we still refer to as $v^*$. Since $\beta^* >0$, $v^*$ has a strict local minimum in 0. By the classification of solutions from Proposition \ref{proposition symmetry and classification}, $v^*$ must be periodic. Moreover, it has a unique local maximum and minimum per period and is symmetric with respect to its extrema.

The uniqueness of $v^*$ up to translations follows from Proposition \ref{proposition (v,v')-uniqueness}. This completes the proof of Theorem \ref{theorem existence}.
\qed

%%%%%%%%%%%%%%%%%%%%%%%%

\subsection{Proof of Theorem \ref{theorem periodic}}

By \cite[Theorem 4.2]{lin}, the positivity of $u$ and the non-removability of the singularity in 0 imply that $u$ is radially symmetric. Since the function $v$ defined by $u(x)= |x|^{-\frac{n-4}{2}} v(\ln |x|)$ satisfies \eqref{eq ODE}, we are in a position to apply the classification result from Proposition \ref{proposition symmetry and classification} and we claim that $v$ is either the constant $B^{\frac{1}{p-1}}=a_0$ or periodic. Indeed, the only case that remains to be excluded is that $v(t) = c_n (2 \cosh(t-T))^{-\frac{n-4}{2}}$. But in this case, it is clear that $v(t) \sim c_n e^{\frac{n-4}{2}t}$ as $t \to -\infty$ and hence the singularity of $u$ would be removable, contradicting the assumptions. Thus, either $v$ is constant or periodic. Let $a:=\inf v$. Then by the first part of Theorem \ref{theorem existence} $a\in (0,a_0]$, and $a=a_0$ if and only if $v\equiv a_0$. Moreover, for $a<a_0$ the function $v$ is periodic with minimal value $a$. Therefore, by the second part of Theorem \ref{theorem existence} $v(t)=v_a(t+L)$ for some $L\in\R$. This completes the proof of Theorem \ref{theorem periodic}.
\qed

%%%%%%%%%%%%%%%%%%%%%%%%

\section{Appendix: Proof of Proposition \ref{proposition (v,v')-uniqueness}}

In this appendix, we give the proof of Proposition \ref{proposition (v,v')-uniqueness}, following and simplifying \cite{vdb}.

Let $v$ and $w$ be bounded solutions of \eqref{eq ODE with f} which satisfy $v(0) = w(0)$ and $v'(0) = w'(0)$. We can assume without loss that $v''(0) \geq w''(0)$ (otherwise exchange $v$ and $w$). We may assume furthermore (up to replacing $v(t)$ and $w(t)$ by $v(-t)$ and $w(-t)$) that $v'''(0) \geq w'''(0)$. 

Suppose, by contradiction, that $v \nequiv w$. Then by uniqueness of ODE solutions, $v^{(k)}(0) \neq w^{(k)}(0)$ for $k=2$ or $k=3$. In both cases, we deduce from our hypotheses on the initial conditions that
\[ v(t) > w(t) \quad \text{on } (0, \sigma) \]
for some sufficiently small $\sigma >0$. 

Recalling that $4B<A^2$ we introduce the positive real numbers
\[ \lambda = \frac{A}{2} - \sqrt{(\frac{A}{2})^2 - B}
\qquad \text{and}\qquad
\mu = \frac{A}{2} + \sqrt{(\frac{A}{2})^2 - B} \]
and define the auxiliary functions 
\[ \phi(t):= v''(t) - \lambda v(t) \qquad \text{and}\qquad \psi(t):= w''(t) - \lambda w(t). \]
Then by the hypotheses, we have
\begin{equation} \label{eq comp 1} (\phi-\psi)(0) \geq 0 \qquad \text{and}\qquad (\phi-\psi)'(0) \geq 0. \end{equation}
We note that $\lambda + \mu = A$ and $\lambda \mu = B$. Therefore equation \eqref{eq ODE with f} for $v$ and $w$ implies that
\[ (\phi-\psi)''(t) - \mu (\phi-\psi)(t) = |v(t)|^{p-1}v(t) - |w(t)|^{p-1} w(t) \]
for all $t \in \R$. Since $v(t) > w(t)$ on $(0, \sigma)$ and since the function $u \mapsto |u|^{p-1}u$ is strictly increasing on $\R$, this implies that
\begin{equation} \label{eq comp 2} (\phi-\psi)''(t) - \mu (\phi-\psi)(t) > 0 \quad \text{ for all   } t \in (0, \sigma). \end{equation}
The inequalities \eqref{eq comp 1} and \eqref{eq comp 2} and the fact that $\mu>0$ easily imply that $(\phi - \psi)(t) \geq 0$ for $t \in (0, \sigma)$, or equivalently, that
\begin{equation} \label{eq comp 3} (v-w)''(t) \geq \lambda (v-w) (t) > 0  \quad \text{ for all   } t \in (0, \sigma). \end{equation}
Since $(v-w)'(0) \geq 0$ by the hypotheses of the lemma and since $\lambda>0$, we see from \eqref{eq comp 3} that $(v-w)'(t) >0$ for all $t \in (0, \sigma)$. Hence $v-w$ is strictly increasing on $(0, \sigma)$ and since $\sigma >0$ was arbitrary with the property that $v-w>0$ on $(0, \sigma)$, we infer that $v-w$ remains strictly positive for all times. 

Repeating the above arguments for the interval $(0, \infty)$ instead of $(0, \sigma)$, we see from \eqref{eq comp 3} that $(v-w)'$ is positive and strictly increasing on $(0, \infty)$. This of course contradicts the boundedness of $v-w$. This proves that in fact we must have $v \equiv w$, concluding the proof of Proposition \ref{proposition (v,v')-uniqueness}.

\bibliography{fourthorder}

\begin{thebibliography}{10}

\bibitem{AmTo}
C.~J. Amick and J.~F. Toland.
\newblock Homoclinic orbits in the dynamic phase-space analogy of an elastic
  strut.
\newblock {\em European J. Appl. Math.}, 3(2):97--114, 1992.

\bibitem{BaRe}
Sami Baraket and Salem Rebhi.
\newblock Construction of dipole type singular solutions for a biharmonic
  equation with critical {S}obolev exponent.
\newblock {\em Adv. Nonlinear Stud.}, 2(4):459--476, 2002.

\bibitem{BuChTo}
B.~Buffoni, A.~R. Champneys, and J.~F. Toland.
\newblock Bifurcation and coalescence of a plethora of homoclinic orbits for a
  {H}amiltonian system.
\newblock {\em J. Dynam. Differential Equations}, 8(2):221--279, 1996.

\bibitem{caffarelli}
Luis~A. Caffarelli, Basilis Gidas, and Joel Spruck.
\newblock Asymptotic symmetry and local behavior of semilinear elliptic
  equations with critical {S}obolev growth.
\newblock {\em Comm. Pure Appl. Math.}, 42(3):271--297, 1989.

\bibitem{ChGuYa}
Sun-Yung~A. Chang, Matthew~J. Gursky, and Paul~C. Yang.
\newblock Regularity of a fourth order nonlinear {PDE} with critical exponent.
\newblock {\em Amer. J. Math.}, 121(2):215--257, 1999.

\bibitem{ChYa}
Sun-Yung~A. Chang and Paul~C. Yang.
\newblock On uniqueness of solutions of {$n$}th order differential equations in
  conformal geometry.
\newblock {\em Math. Res. Lett.}, 4(1):91--102, 1997.

\bibitem{wei}
Z.~{Guo}, X.~{Huang}, L.~{Wang}, and J.~{Wei}.
\newblock {On Delaunay solutions of a biharmonic elliptic equation with
  critical exponent}.
\newblock {\em ArXiv e-prints}, August 2017.

\bibitem{GuWeZh}
Zongming Guo, Juncheng Wei, and Feng Zhou.
\newblock Singular radial entire solutions and weak solutions with prescribed
  singular set for a biharmonic equation.
\newblock {\em J. Differential Equations}, 263(2):1188--1224, 2017.

\bibitem{HaYa}
Fengbo Hang and Paul~C. Yang.
\newblock Lectures on the fourth-order {$Q$} curvature equation.
\newblock In {\em Geometric analysis around scalar curvatures}, volume~31 of
  {\em Lect. Notes Ser. Inst. Math. Sci. Natl. Univ. Singap.}, pages 1--33.
  World Sci. Publ., Hackensack, NJ, 2016.

\bibitem{korevaar}
Nick Korevaar, Rafe Mazzeo, Frank Pacard, and Richard Schoen.
\newblock Refined asymptotics for constant scalar curvature metrics with
  isolated singularities.
\newblock {\em Invent. Math.}, 135(2):233--272, 1999.

\bibitem{Lieb}
Elliott~H. Lieb.
\newblock Sharp constants in the {H}ardy-{L}ittlewood-{S}obolev and related
  inequalities.
\newblock {\em Ann. of Math. (2)}, 118(2):349--374, 1983.

\bibitem{lin}
Chang-Shou Lin.
\newblock A classification of solutions of a conformally invariant fourth order
  equation in {${\bf R}^n$}.
\newblock {\em Comment. Math. Helv.}, 73(2):206--231, 1998.

\bibitem{mazzeo-pacard}
Rafe Mazzeo and Frank Pacard.
\newblock Constant scalar curvature metrics with isolated singularities.
\newblock {\em Duke Math. J.}, 99(3):353--418, 1999.

\bibitem{peletier}
L.~A. Peletier and W.~C. Troy.
\newblock {\em Spatial patterns}, volume~45 of {\em Progress in Nonlinear
  Differential Equations and their Applications}.
\newblock Birkh\"auser Boston, Inc., Boston, MA, 2001.
\newblock Higher order models in physics and mechanics.

\bibitem{schoen}
Richard~M. Schoen.
\newblock The existence of weak solutions with prescribed singular behavior for
  a conformally invariant scalar equation.
\newblock {\em Comm. Pure Appl. Math.}, 41(3):317--392, 1988.

\bibitem{schoen2}
Richard~M. Schoen.
\newblock Variational theory for the total scalar curvature functional for
  {R}iemannian metrics and related topics.
\newblock In {\em Topics in calculus of variations ({M}ontecatini {T}erme,
  1987)}, volume 1365 of {\em Lecture Notes in Math.}, pages 120--154.
  Springer, Berlin, 1989.

\bibitem{UhVi}
Karen~K. Uhlenbeck and Jeff~A. Viaclovsky.
\newblock Regularity of weak solutions to critical exponent variational
  equations.
\newblock {\em Math. Res. Lett.}, 7(5-6):651--656, 2000.

\bibitem{vdb}
Jan~Bouwe van~den Berg.
\newblock The phase-plane picture for a class of fourth-order conservative
  differential equations.
\newblock {\em J. Differential Equations}, 161(1):110--153, 2000.

\bibitem{wei-xu}
Juncheng Wei and Xingwang Xu.
\newblock Classification of solutions of higher order conformally invariant
  equations.
\newblock {\em Math. Ann.}, 313(2):207--228, 1999.

\end{thebibliography}
	\bibliographystyle{plain}
\end{document}